\documentclass[12pt]{amsart}

\usepackage[a4paper,margin=1in]{geometry}
\usepackage[T1]{fontenc}
\usepackage{amsmath,amssymb,amsthm,mathtools}
\usepackage{microtype}
\usepackage{xcolor}
\definecolor{linkblue}{RGB}{0,82,155}
\usepackage[
  colorlinks=true,
  linkcolor=linkblue,
  citecolor=linkblue,
  urlcolor=linkblue
]{hyperref}
\usepackage[nameinlink,capitalise,noabbrev]{cleveref}
\usepackage{enumitem}

\crefname{theorem}{Theorem}{Theorems}
\Crefname{theorem}{Theorem}{Theorems}

\crefname{lemma}{Lemma}{Lemmas}
\Crefname{lemma}{Lemma}{Lemmas}

\crefname{corollary}{Corollary}{Corollaries}
\Crefname{corollary}{Corollary}{Corollaries}

\crefname{proposition}{Proposition}{Propositions}
\Crefname{proposition}{Proposition}{Propositions}

\crefname{remark}{Remark}{Remarks}
\Crefname{remark}{Remark}{Remarks}

\crefname{definition}{Definition}{Definitions}
\Crefname{definition}{Definition}{Definitions}

\selectfont
\setlist{itemsep=0.25em,topsep=0.4em}

\usepackage{aliascnt}

\newtheorem{theorem}{Theorem}[section]

\newaliascnt{lemma}{theorem}
\newtheorem{lemma}[lemma]{Lemma}
\aliascntresetthe{lemma}

\newaliascnt{corollary}{theorem}

\aliascntresetthe{corollary}

\newaliascnt{proposition}{theorem}
\newtheorem{proposition}[proposition]{Proposition}
\aliascntresetthe{proposition}

\newtheoremstyle{myremark}%
  {6pt}%
  {6pt}%
  {\normalfont}%
  {}%
  {\bfseries}%
  {.}%
  {.5em}%
  {}

\theoremstyle{myremark}
\newaliascnt{remark}{theorem}
\newtheorem{remark}[remark]{Remark}
\aliascntresetthe{remark}

\newaliascnt{definition}{theorem}
\newtheorem{definition}[definition]{Definition}
\aliascntresetthe{definition}

\newcommand{\set}[1]{\left\{#1\right\}}
\newcommand{\abs}[1]{\left|#1\right|}
\newcommand{\floor}[1]{\left\lfloor #1 \right\rfloor}
\newcommand{\ceil}[1]{\left\lceil #1 \right\rceil}

\title{Pseudoshattering Pairs}

\author{Noga Alon}
\address{Department of Mathematics, Princeton University, 
Princeton, New Jersey, USA}
\email{nalon@math.princeton.edu}
\thanks{Research supported in part by
NSF grants DMS-2154082 and DMS-2349013}
\author{Varun Sivashankar}
\address{Department of Mathematics, Princeton University, 
Princeton, New Jersey, USA}
\email{varunsiva@princeton.edu}
\date{}

\begin{document}
\begin{abstract}
For two vectors $x,y\in [b]^k$, consider the bipartite graph with two copies of
$[b]$ in which $i$ on the left is joined to $j$ on the right if
$(x_t,y_t)=(i,j)$ for some coordinate $t$.  We study the largest size of a family
$C\subseteq [b]^k$ such that, for every two distinct $x,y\in C$, this bipartite
graph contains a cycle.

We give a natural construction for such families and conjecture that it is
optimal whenever $k$ is large relative to $b$.  We prove an LYM-type upper bound
that is asymptotically tight with respect to this construction, and is exact
when $k$ is large and divisible by $b$.  We then refine the argument using a
circular ordering, obtaining the sharp full-support bound when
$k\equiv -1\pmod b$.  In the case $b=3$, we prove the exact general result when
$k\equiv -1\pmod 3$ and $k$ is sufficiently large.  The problem is motivated by
the Daniely--Shalev-Shwartz dimension and the pseudocube formulation of a
higher-alphabet Sauer--Shelah--Perles lemma.
\end{abstract}

\maketitle

\section{Introduction}

Let $[m]=\{1,\dots,m\}$. For a family $\mathcal{F}\subseteq 2^{[n]}$, a set $S\subseteq [n]$ is shattered by $\mathcal{F}$ if
\[
\set{A\cap S:A\in\mathcal{F}}=2^S.
\]
Equivalently, if one identifies $\mathcal{F}$ with a binary matrix whose rows are indicator vectors,
then a set of columns is shattered if the rows of the
corresponding submatrix contain every binary pattern.
The Sauer--Shelah lemma \cite{Sauer,Shelah}, in the equivalent formulation due to Pajor \cite{Pajor},
states that $\mathcal{F}$ shatters at least $\abs{\mathcal{F}}$ sets.

Recently Hanneke, Meng, Moran, and Shaeiri \cite{HannekeMengMoranShaeiri} introduced a
$b$-ary generalization based on pseudocubes.  In this paper we focus on the maximum number of \emph{pairs} of columns that are pseudo-shattered.
The relevant combinatorial parameter in \cite{HannekeMengMoranShaeiri} is the
Daniely--Shalev-Shwartz dimension, introduced by Daniely and Shalev-Shwartz
\cite{DanielyShalevShwartz}.
In the classical binary setting, the dual extremal problem of maximizing the number of shattered
sets under a fixed row budget was studied by Das and M\'esz\'aros \cite{DasMeszaros}
and later by Alon, Sivashankar, and Zhu \cite{AlonSivashankarZhu}.
In particular, the exact maximum for pairs may be deduced by combining a Tur\'an reduction
with the theorem of Kleitman and Spencer on pairwise independent set systems \cite{KleitmanSpencer}.
Our problem is the natural $b$-ary pseudocube analogue of that line of work.

Following \cite{HannekeMengMoranShaeiri}, a nonempty set
$S\subseteq [b]^d$ is called an $\ell$-pseudocube
if every point of $S$ has at least $\ell$ neighbors in every coordinate direction.
Given a $k\times n$ matrix over $[b]$, a set of $d$ columns is $\ell$-pseudo-shattered if the set
of row patterns appearing on those columns contains an $\ell$-pseudocube.
We restrict throughout to the case $d=2$.

Let us focus on the case $\ell = 1$.  Fix a $k\times n$ matrix
$M=(m_{t,j})$ over $[b]$.  For a pair of columns $u,v\in [n]$, write
\[
R_{u,v}(M)=\set{(m_{t,u},m_{t,v}):t\in [k]}\subseteq [b]^2.
\]
The pair $\{u,v\}$ is pseudo-shattered exactly when $R_{u,v}(M)$ contains a
$1$-pseudocube.  For any $R\subseteq [b]^2$, let $G(R)$ be the bipartite graph
with left and right vertex classes both identified with $[b]$, and with an edge
$ij$ whenever $(i,j)\in R$.  If $S\subseteq R$ is a $1$-pseudocube, then every
edge of $G(S)$ has another edge with the same left endpoint and another edge with
the same right endpoint.  Thus every nonisolated vertex of $G(S)$ has degree at
least $2$, so $G(S)$ contains a cycle, and therefore $G(R)$ contains a cycle.
Conversely, the edge set of any cycle in $G(R)$ is a $1$-pseudocube.  Hence, in
dimension two, a pair of columns is $1$-pseudo-shattered if and only if its
support graph contains a cycle.

Interestingly, this property was studied in the context of phylogeny where pairwise compatibility of multi-state characters is expressed by the acyclicity
of a partition-intersection graph; see, for example,
Gysel, Lam, and Gusfield \cite{GyselLamGusfield} and
Shutters, Vakati, and Fern\'andez-Baca \cite{ShuttersVakatiFernandezBaca}.
At the other extreme, one may require every pair of $b$-ary columns to realize all ordered pairs
of symbols.
This stronger qualitative-independence condition was studied by Gargano, K\"orner, and Vaccaro
\cite{GarganoKornerVaccaro}, as well as by Poljak, Pultr, and R\"odl \cite{PoljakPultrRodl}.

The cycle characterization leads to a clean extremal reformulation.
For $x,y\in [b]^k$, write
\[
G(x,y)=G(\set{(x_t,y_t):t\in [k]}).
\]
We say that $x$ and $y$ form a \emph{cyclic pair} if $G(x,y)$ contains a cycle.
For integers $b,k\ge 2$, let $m_b(k)$ denote the maximum size of a family
$C\subseteq [b]^k$ such that every two distinct members of $C$ form a cyclic
pair.  Let $F_b(n,k)$ denote the maximum number of pseudo-shattered pairs among
the columns of a $k\times n$ matrix over $[b]$.  The passage from
$F_b(n,k)$ to $m_b(k)$ is just Tur\'an's theorem.  Given a matrix $M$, form a
graph $H(M)$ on its columns, joining two columns when they are pseudo-shattered.
If $H(M)$ contained a clique of size $m_b(k)+1$, then the corresponding column
types would give a family in $[b]^k$ of size $m_b(k)+1$ in which every two
members form a cyclic pair, contradicting the definition of $m_b(k)$.  Hence
$H(M)$ is $K_{m_b(k)+1}$-free, and Tur\'an's theorem gives
$e(H(M))\le t(n,m_b(k))$.  Conversely, if we take a largest cyclic clique in
$[b]^k$ and repeat its words as evenly as possible as columns, then every two
columns from different parts are pseudo-shattered.  This realizes the Tur\'an
graph, and therefore
\[
F_b(n,k)=t(n,m_b(k)),
\]
where $t(n,r)$ is the number of edges in the Tur\'an graph $T(n,r)$.  Thus the
main question is to understand the largest possible cyclic clique, namely
$m_b(k)$.

In \Cref{sec:good} we give a candidate extremal construction and prove that it
forms a cyclic clique, giving a lower bound $N_b(k)$ for $m_b(k)$.  We believe
that this is the exact answer whenever $k$ is sufficiently large relative to
$b$.

The next two sections give evidence for this conjecture 
from two upper-bound
arguments.  In \Cref{sec:upper} we prove a path-based
permutation-packing LYM-type
inequality.  This bound applies to 
arbitrary cliques.  It already has the right
asymptotic size with respect to the 
construction in \Cref{sec:good}, and it is
exact in the divisible case $k=bq$ under the explicit threshold
$(q+1)^q/q!>b$.

In \Cref{sec:circular} we replace the path by a cycle to get a sharper
fixed-support bound.  For full-support cliques this 
circular bound matches the
good construction when 
$k\equiv 0,-1\pmod b$.  Summing the fixed-support bounds
over all possible 
support sizes gives a general clique bound that improves the path bound
eventually, but still has additive lower-support terms.  
We expect those terms
to be artifacts.  As evidence, we prove in the specific case $b=3$ and
$k\equiv -1\pmod 3$ that the full-support bound already controls an arbitrary
clique, giving the exact value of $m_3(k)$ in that residue class for large $k$.

\section{A Lower Bound Construction}\label{sec:good}

Assume $k\ge b$, and write $k=bq+r$, where $0\le r<b$.
Set
\[
s_1=\cdots=s_{b-r}=q,\qquad s_{b-r+1}=\cdots=s_b=q+1.
\]

\begin{definition}
For a vector $v\in [b]^k$, let
\[
A_i(v)=\set{t\in [k]:v_t=i}.
\]
We say that $v$ is \emph{good} if
\[
\abs{A_i(v)}=s_i\qquad\text{for all }i\in [b],
\]
and the first occurrences of the symbols are in increasing order:
\[
\min A_1(v)<\min A_2(v)<\cdots<\min A_b(v).
\]
Let $N_b(k)$ denote the number of good vectors in $[b]^k$.
\end{definition}

\begin{theorem}\label{thm:good-clique}
Any two distinct good vectors in $[b]^k$ form a cyclic pair.  In particular,
\[
m_b(k)\ge N_b(k).
\]
\end{theorem}

\begin{proof}
Let $v,w\in [b]^k$ be distinct good vectors.
Write
\[
A_i=A_i(v),\qquad B_i=A_i(w).
\]
Suppose for contradiction that $G(v,w)$ is a forest.

First remove all isolated edge components from $G(v,w)$.
An isolated edge component $A_iB_j$ means that $A_i$ meets only $B_j$ and $B_j$ meets only $A_i$,
which forces $A_i=B_j$ because both $\{A_1,\dots,A_b\}$ and $\{B_1,\dots,B_b\}$ are partitions of $[k]$.
If every edge component were isolated, then the two partitions would coincide.
Since both vectors are good and label their blocks by increasing first occurrence, this would imply
$A_i=B_i$ for every $i$, hence $v=w$, a contradiction.
So after removing isolated edge components, a nonempty forest remains; call it $H$.

If $r=0$, then every block has size $q$.
Let $A_i$ be a leaf of $H$, and let $B_j$ be its unique neighbor.
Then $A_i\subseteq B_j$, so $\abs{A_i}\le \abs{B_j}$.
Equality is impossible because isolated edges were removed, and strict inequality is impossible
because both sizes are $q$.
This contradiction shows that $H$ cannot exist when $r=0$.
Hence we may assume $r>0$ from now on.

Now let $A_i$ be a leaf of $H$, with unique neighbor $B_j$.
Since $A_i$ meets no other block on the $B$-side, we have $A_i\subseteq B_j$.
Thus $\abs{A_i}\le \abs{B_j}$.
Equality would again force $A_i=B_j$, giving an isolated edge, so in fact
$\abs{A_i}<\abs{B_j}$.
Because the only block sizes are $q$ and $q+1$, we obtain
\[
\abs{A_i}=q,\qquad \abs{B_j}=q+1,
\]
and hence
\[
i\le b-r<j.
\]
Similarly, if $B_\ell$ is a leaf of $H$ with unique neighbor $A_h$, then
\[
\abs{B_\ell}=q,\qquad \abs{A_h}=q+1,
\]
so
\[
\ell\le b-r<h.
\]

We next claim that $H$ has a leaf on each side.
Suppose, for example, that every leaf of $H$ lay on the $A$-side.
Then every $B$-vertex of $H$ would have degree at least $2$.
In any tree component of $H$ with $a$ vertices on the $A$-side and $c$ vertices on the $B$-side,
the number of edges is $a+c-1$.
Summing degrees over the $B$-side therefore gives
\[
a+c-1\ge 2c,
\]
so $a>c$ in each component.
Summing over all components shows that $H$ has strictly more vertices on the $A$-side than on the
$B$-side, which is impossible because we removed isolated edge components in matched pairs.
Thus $H$ has a leaf on the $B$-side as well.
By symmetry, it also has a leaf on the $A$-side.

Choose an $A$-side leaf $A_i$ with unique neighbor $B_j$, and a $B$-side leaf $B_\ell$
with unique neighbor $A_h$.
From the inequalities above we have
\[
i\le b-r<j,\qquad \ell\le b-r<h,
\]
so in particular
\[
i<h,\qquad \ell<j.
\]

Let
\[
\alpha_t=\min A_t,\qquad \beta_t=\min B_t.
\]
Since $v$ and $w$ are good,
\[
\alpha_1<\alpha_2<\cdots<\alpha_b,\qquad
\beta_1<\beta_2<\cdots<\beta_b.
\]
Because $A_i\subseteq B_j$, we have $\beta_j\le \alpha_i$.
Because $B_\ell\subseteq A_h$, we have $\alpha_h\le \beta_\ell$.
Using $i<h$ and $\ell<j$, we get
\[
\alpha_i<\alpha_h,\qquad \beta_\ell<\beta_j.
\]
Putting everything together yields
\[
\beta_j\le \alpha_i<\alpha_h\le \beta_\ell<\beta_j,
\]
a contradiction.
Therefore $G(v,w)$ cannot be a forest, so it contains a cycle.
\end{proof}

\begin{proposition}\label{prop:count-good}
The number of good vectors is
\[
N_b(k)=\prod_{i=1}^{b-1}\binom{k-\sum_{j<i}s_j-1}{s_i-1}.
\]
In particular, if $k=bq$, then
\[
N_b(k)=\frac{k!}{b!(q!)^b}.
\]
\end{proposition}

\begin{proof}
We construct a good vector block by block.
The first block $A_1$ must contain $1$, so there are
\[
\binom{k-1}{s_1-1}
\]
choices for $A_1$.
After choosing $A_1,\dots,A_{i-1}$, let $m_i$ be the smallest unused element.
The condition of being good forces $m_i\in A_i$, and we may choose the remaining
$s_i-1$ elements of $A_i$ arbitrarily from the unused elements larger than $m_i$.
At that stage there are
\[
k-\sum_{j<i}s_j-1
\]
available choices, so the number of possibilities for $A_i$ is
\[
\binom{k-\sum_{j<i}s_j-1}{s_i-1}.
\]
Multiplying over $i=1,\dots,b-1$ gives the formula.

If $k=bq$, then $s_1=\cdots=s_b=q$, so the product formula gives
\[
N_b(k)=\prod_{i=1}^{b-1}\binom{k-(i-1)q-1}{q-1}.
\]
A straightforward telescoping simplification yields
\[
N_b(k)=\frac{(k-1)!}{(q-1)!^b\prod_{t=1}^{b-1}(k-tq)}
=\frac{k!}{b!(q!)^b}.
\]
\end{proof}

\section{A Permutation-Packing Upper Bound}\label{sec:upper}

In this section we give an upper bound in the spirit of Meshalkin's theorem \cite{Meshalkin}.

For a word $x\in [b]^k$, write
\[
k_i(x)=\abs{\set{t\in [k]:x_t=i}},
\]
and
\[
b'(x)=\abs{\set{i\in [b]:k_i(x)>0}}.
\]
An \emph{interval partition} of $[k]$ is a partition whose blocks can be ordered
$I_1,\dots,I_s$ so that, for some
\[
1=a_1<a_2<\cdots<a_{s+1}=k+1,
\]
we have
\[
I_j=\{a_j,a_j+1,\dots,a_{j+1}-1\}\qquad\text{for }j=1,\dots,s.
\]

\begin{lemma}\label{lem:interval-forest}
Let $P$ and $Q$ be two interval partitions of $[k]$.
Then the bipartite overlap graph with vertex classes $P$ and $Q$ and edges given by nonempty intersection
is a forest.
\end{lemma}

\begin{proof}
Order the blocks as
\[
P=\{I_1,\dots,I_s\},\qquad Q=\{J_1,\dots,J_t\},
\]
from left to right.
We argue by induction on $s+t$.
The claim is trivial when $s+t=2$.

Both $I_1$ and $J_1$ contain the element $1$.
If $\max I_1\le \max J_1$, then $I_1\subseteq J_1$, so $I_1$ is a leaf in the overlap graph.
Delete the vertex $I_1$ and remove its points from $J_1$, deleting $J_1$ as well if it becomes empty.
After translating the remaining ground set, we obtain two interval partitions of a smaller set,
so the remaining overlap graph is a forest by induction.
Reattaching the leaf $I_1$ shows that the original overlap graph is also a forest.

The case $\max J_1\le \max I_1$ is symmetric, with $J_1$ as the leaf.
\end{proof}

\begin{theorem}[An LYM-type inequality]\label{thm:packing}
Let $C\subseteq [b]^k$ be a clique.  For $x\in C$, let $\Pi(x)$ be the set of
permutations $\sigma\in S_k$ such that
\[
x_{\sigma^{-1}(1)}x_{\sigma^{-1}(2)}\cdots x_{\sigma^{-1}(k)}
\]
is constant on $b'(x)$ consecutive intervals, one interval for each symbol that
appears in $x$.  Then the sets $\Pi(x)$, $x\in C$, are pairwise disjoint.
Consequently,
\[
\sum_{x\in C} b'(x)!\prod_{i=1}^b k_i(x)!\le k!.
\]
\end{theorem}

\begin{proof}
Suppose $\sigma\in \Pi(x)\cap \Pi(y)$ for some distinct $x,y\in C$.
After applying $\sigma$, both words become constant on consecutive intervals.
Therefore the partitions of $[k]$ induced by the level sets of $x$ and $y$ become interval partitions.
By \Cref{lem:interval-forest}, their overlap graph is a forest.
But this overlap graph is exactly $G(x,y)$, contradicting the assumption that $x$ and $y$
form a cyclic pair.
Thus $\Pi(x)\cap \Pi(y)=\varnothing$ whenever $x\neq y$ in $C$.

Since the sets $\Pi(x)$ are pairwise disjoint subsets of $S_k$, we have
\[
\sum_{x\in C}\abs{\Pi(x)}\le k!.
\]
For fixed $x$, choose an ordering of the $b'(x)$ symbols appearing in $x$, and
then order the elements inside each level set.  This gives
\[
\abs{\Pi(x)}=b'(x)!\prod_{i=1}^b k_i(x)!,
\]
so the displayed inequality gives the result.
\end{proof}

\begin{theorem}\label{thm:path-consequences}
Let $C\subseteq [b]^k$ be a clique, and write $k=bq+r$, where $0\le r<b$.
Then the following hold.
\begin{enumerate}[label=\textup{(\roman*)}]
\item
Every clique satisfies
\[
\abs{C}\le \frac{k!}{\min_{x\in C}\left(b'(x)!\prod_{i=1}^b k_i(x)!\right)}.
\]
\item
If
\[
\frac{(q+1)^q}{q!}>b,
\]
then every clique satisfies
\[
\abs{C}\le
\frac{k!}{b!(q!)^{b-r}((q+1)!)^r}.
\]
\item
Under the same hypothesis,
\[
m_b(k)\le (1+O(b/q))N_b(k).
\]
In particular, for fixed $b$,
\[
m_b(k)=(1+o(1))N_b(k)
\]
as $q\to\infty$.
\item
If the hypothesis in \textup{(ii)} holds and $r=0$, then
\[
m_b(k)=\frac{k!}{b!(q!)^b}.
\]
\item
Let $b^*$ be a value of the integer $y$ which attains the minimum
of the expression $y! \prod_{i=1}^y k_i!$ over all integers
$1\leq y\leq k$ and positive integers $k_i$ satisfying
$\sum_{i=1}^y k_i=k$.  Suppose $k=b^*q+r$ with $0 \leq r < b^*$.
Then, for every $b\geq 2$,
\[
m_b(k) \leq \frac{k!}{b^*!(q!)^{b^*-r}((q+1)!)^r}.
\]
If $r=0$, then for every $b\geq b^*$,
\[
m_b(k)=\frac{k!}{b^*!(q!)^{b^*}}.
\]
\end{enumerate}
\end{theorem}

\begin{proof}
Put
\[
D(x)=b'(x)!\prod_{i=1}^b k_i(x)!.
\]
Item \textup{(i)} follows immediately from \Cref{thm:packing}, since
\[
\abs{C}\min_{x\in C}D(x)\le \sum_{x\in C}D(x)\le k!.
\]

Fix the support size $s=b'(x)$, and let $a_1,\dots,a_s$ be the positive
multiplicities of symbols in $x$.  If $a_j\ge a_i+2$, then replacing
$(a_i,a_j)$ by $(a_i+1,a_j-1)$ changes the product of factorials by the factor
\[
\frac{(a_i+1)!(a_j-1)!}{a_i!a_j!}
=\frac{a_i+1}{a_j}<1.
\]
Thus, for fixed $s$, $D(x)$ is minimized when the $s$ positive multiplicities
are as equal as possible.

The full-support balanced value is
\[
D_b=b!(q!)^{b-r}((q+1)!)^r=b!(q!)^b(q+1)^r.
\]
For $s<b$, the balanced $s$-support multiplicities are all at least $q$, so
write them as $q+e_1,\dots,q+e_s$, where $e_i\ge0$ and
\[
e_1+\cdots+e_s=(b-s)q+r.
\]
Then the corresponding denominator is at least
\[
s!(q!)^s(q+1)^{(b-s)q+r}.
\]
Dividing by $D_b$, we get
\[
\frac{s!(q!)^s(q+1)^{(b-s)q+r}}{D_b}
=
\frac{s!}{b!}
\left(\frac{(q+1)^q}{q!}\right)^{b-s}
>
\frac{s!}{b!}b^{b-s}\ge 1.
\]
Hence $D(x)\ge D_b$ for every word $x$, and item \textup{(ii)} follows from
item \textup{(i)}.

It remains to compare this bound with the construction.  Let
\[
s_1=\cdots=s_{b-r}=q,\qquad s_{b-r+1}=\cdots=s_b=q+1,
\qquad R_i=s_i+\cdots+s_b.
\]
By \Cref{prop:count-good},
\[
N_b(k)=
\frac{(k-1)!}
{s_b!\prod_{i=1}^{b-1}(s_i-1)!\prod_{i=2}^{b-1}R_i}.
\]
Therefore
\[
\frac{k!}{b!(q!)^{b-r}((q+1)!)^r\,N_b(k)}
=
\frac{k}{b!}\,
\frac{\prod_{i=2}^{b-1}R_i}{\prod_{i=1}^{b-1}s_i}
\]
\[
\leq \frac{k}{b!} \frac{(b-1)!(q+1)^{b-2}}{q^{b-1}}
\leq \frac{k}{bq}(1+O(\frac{b}{q}))\leq 1+O(\frac{b}{q}).
\]
Together with the lower bound $m_b(k)\ge N_b(k)$, this proves item
\textup{(iii)}.  When $r=0$, \Cref{prop:count-good} gives
$N_b(k)=k!/(b!(q!)^b)$, so the upper bound is tight, proving item
\textup{(iv)}.

For item \textup{(v)}, the definition of $b^*$ implies that every
word $x\in [b]^k$ satisfies
\[
b'(x)!\prod_{i=1}^b k_i(x)!
\geq b^*!(q!)^{b^*-r}((q+1)!)^r.
\]
The upper bound follows from item \textup{(i)}.  If $r=0$, the construction
using $b^*$ symbols has size $k!/(b^*!(q!)^{b^*})$ and embeds in $[b]^k$
for every $b\geq b^*$.  Hence the upper bound is attained.
\end{proof}

\begin{remark}
There are only some values of $k$ for which the minimizer $b^*$ is not unique.
In fact, if
\[
F_k(y)=y!(q!)^{y-r}((q+1)!)^r,\qquad k=yq+r,\quad 0\le r<y,
\]
then $F_k(y)$ is unimodal in $y$, so nonuniqueness forces an adjacent tie.  A
direct adjacent-ratio calculation shows that the only such values are
\[
k\in\{2,5,16,23,223,269\}.
\]
The divisible case in item \textup{(v)} can also be recognized explicitly.  If
$k=Qb^*$, then $b^*$ is a minimizer precisely when
\[
\max\left\{1,\left\lceil \frac{Q^Q}{Q!}-1\right\rceil\right\}
\le b^*\le
\left\lfloor \frac{(Q+1)^Q}{Q!}\right\rfloor .
\]
This follows by comparing $F_k(b^*)$ with its two neighboring values; the
unimodality of $F_k$ makes these local inequalities sufficient.  We omit the
details for brevity.
\end{remark}

\section{Improving the Upper Bound}\label{sec:circular}

The permutation-packing bound in \Cref{sec:upper} applies to arbitrary cliques.
In particular, it applies to the special case where every word has full support.
If $k=bq+r$ with $0\le r<b$, this full-support specialization is
\[
P_{b,r}(q):=
\frac{k!}{b!(q!)^{b-r}((q+1)!)^r}.
\]
The circular argument improves this fixed-support bound by replacing the path
with a cycle.  This use of cyclic orders is an adaptation of Katona's cycle
method, introduced in his proof of the Erd\H{o}s--Ko--Rado theorem and developed
more broadly as a permutation method in extremal set theory
\cite{KatonaEKR,KatonaCycleMethod,ErdosFurediKatona,FranklKatonaCircle}.  The
extra structure here is the set of boundary edges between consecutive level
intervals.  If two clique members are interval words for the same cyclic order
and share a boundary edge, then cutting the circle at that edge turns their
level sets into two interval partitions of a line.  The interval forest lemma
then says that their overlap graph cannot contain a cycle.  Thus, for a fixed
cyclic order, the boundary sets of clique members are disjoint.

For full-support words, the circular bound is $(bq/k)P_{b,r}(q)$.  Thus it is
strictly smaller than the path bound whenever $r>0$.  Since the canonical
construction itself uses full support, this is the right comparison to make; the
circular full-support bound matches the construction when $r=0$ and when
$r=b-1$.  To get back to arbitrary cliques, one can sum the circular bounds over
the possible support sizes.  This gives a valid general upper bound and
eventually improves the general path bound, but it is probably not the truth:
the additive lower-support terms should not be necessary.  As evidence, we prove
that when $b=3$ and $k\equiv2\pmod3$, the full-support circular bound already
upper-bounds arbitrary cliques.  The proof is in \Cref{app:b3-support-two}.

\subsection{The circular support bound}

An \emph{oriented cyclic order} $\Omega=(v_1,\dots,v_k)$ of $[k]$ is a cyclic
listing, considered up to rotation but not reversal.  There are $(k-1)!$ such
orders.  Indices are read modulo $k$, and the cycle-edges are
\[
e_j=\{v_j,v_{j+1}\},\qquad j\in[k].
\]
A nonempty set of the form
\[
\{v_a,v_{a+1},\dots,v_{a+L-1}\},\qquad L\ge 1,
\]
is a \emph{cyclic interval} of $\Omega$; in particular $[k]$ itself is allowed.

For $x\in [b]^k$, write
\[
A_i(x)=\{s\in[k]:x_s=i\}.
\]
We say that $x$ is an \emph{$\Omega$-interval word} if every nonempty level set
$A_i(x)$ is a cyclic interval of $\Omega$.  Its boundary is
\[
\partial_\Omega x
=
\{e_j:x_{v_j}\ne x_{v_{j+1}}\}.
\]
Thus an $\Omega$-interval word using $b'(x)$ symbols has
\[
\abs{\partial_\Omega x}=b'(x).
\]

\begin{lemma}\label{lem:circular-boundaries-disjoint}
Fix an oriented cyclic order $\Omega$ of $[k]$.  Let $C\subseteq [b]^k$ be a clique.
If $x,y\in C$ are distinct $\Omega$-interval words, then
\[
\partial_\Omega x\cap \partial_\Omega y=\varnothing.
\]
\end{lemma}

\begin{proof}
Suppose that $e\in \partial_\Omega x\cap \partial_\Omega y$.  Cut the cycle
$\Omega$ at the edge $e$.  Since $e$ is a boundary edge for both words, no nonempty
level set of either word wraps around the cut.  Hence the nonempty level sets of
$x$ and $y$ become interval partitions of a linearly ordered set.  By
\Cref{lem:interval-forest}, their overlap graph is a forest.  This overlap graph is
exactly $G(x,y)$ after deleting isolated symbol vertices, contradicting the
assumption that $C$ is a clique.
\end{proof}

For $2\le t\le b$, let
\[
C_t=\{x\in C:b'(x)=t\}
\]
be the part of a clique consisting of words that use exactly $t$ symbols.

\begin{theorem}[Circular support-size bound]\label{thm:circular-support}
Let $C\subseteq [b]^k$ be a clique, and fix $2\le t\le b$.  Write
\[
k=tQ+\rho,\qquad 0\le \rho<t.
\]
Then
\[
\abs{C_t}\le
\frac{Q(k-1)!}{(t-1)!\,(Q!)^{t-\rho}((Q+1)!)^\rho}.
\]
\end{theorem}

\begin{proof}
Let
\[
\mathcal I_t=
\{(\Omega,x):\Omega\text{ is an oriented cyclic order of }[k],\
x\in C_t,\text{ and }x\text{ is an }\Omega\text{-interval word}\}.
\]
For fixed $\Omega$, the sets $\partial_\Omega x$ are pairwise disjoint
$t$-subsets of the $k$ cycle-edges, so there are at most $Q=\floor{k/t}$ possible
words $x$.  Hence
\[
\abs{\mathcal I_t}
\le Q(k-1)!.
\]

On the other hand, if the positive multiplicities of $x\in C_t$ are
\[
a_1,\dots,a_t,\qquad a_1+\cdots+a_t=k,
\]
then $x$ is an $\Omega$-interval word for exactly
\[
(t-1)!\prod_{i=1}^t a_i!.
\]
Indeed, choose the cyclic order of the $t$ labelled level sets and then order
the elements inside each level set.
By log-convexity of the factorial function,
\[
\prod_{i=1}^t a_i!\ge (Q!)^{t-\rho}((Q+1)!)^\rho,
\]
and the double count gives the result.
\end{proof}

\subsection{Full-support cliques}

Write $k=bq+r$, where $0\le r<b$ and $q\ge 1$, and let $C_b$ be the
full-support part of a clique $C\subseteq [b]^k$.  Applying
\Cref{thm:circular-support} with $t=b$ gives
\[
\abs{C_b}\le
U_{b,r}(q):=
\frac{q(k-1)!}{(b-1)!(q!)^{b-r}((q+1)!)^r}.
\]
The path bound for $C_b$ alone is $P_{b,r}(q)$, so this improves the
full-support estimate by the factor $bq/k$.

We compare this with the size of the canonical construction.  Let
\[
s_1=\cdots=s_{b-r}=q,\qquad
s_{b-r+1}=\cdots=s_b=q+1
\]
be the canonical multiplicities, and put
\[
R_i=s_i+s_{i+1}+\cdots+s_b.
\]
By \Cref{prop:count-good},
\[
N_b(k)=
\prod_{i=1}^{b-1}
\binom{R_i-1}{s_i-1}
=
\frac{(k-1)!}
{s_b!\prod_{i=1}^{b-1}(s_i-1)!\prod_{i=2}^{b-1}R_i}.
\]
Thus
\[
\frac{U_{b,r}(q)}{N_b(k)}
=
\frac{q}{(b-1)!}
\frac{\prod_{i=2}^{b-1}R_i}{\prod_{i=1}^{b-1}s_i}.
\]
In particular,
\[
U_{b,0}(q)=N_b(k)\qquad\text{when }r=0
\]
and
\[
U_{b,b-1}(q)=N_b(k)\qquad\text{when }r=b-1,
\]
so the circular full-support bound is exact in the divisible residue class and
in the top residue class.  For the remaining residues, the loss is explicit:
if $1\le r\le b-1$, then
\[
\frac{U_{b,r}(q)}{N_b(k)}
=
\frac{r!}{(b-1)!}
\prod_{m=r+1}^{b-1}\left(m+\frac rq\right).
\]
This factor is $1$ when $r=b-1$, is greater than $1$ when $1\le r\le b-2$, and
records exactly how far the circular full-support bound remains from the
construction in the other residue classes.

\subsection{Arbitrary cliques by summing supports}

Using the fixed-support circular bound for each possible support size gives a
bound for unrestricted cliques.  For
$2\le t\le b$, write
\[
k=tQ_t+\rho_t,\qquad 0\le \rho_t<t,
\]
and set
\[
B_t(k)=
\frac{Q_t(k-1)!}{(t-1)!\,(Q_t!)^{t-\rho_t}((Q_t+1)!)^{\rho_t}}.
\]
Then every clique $C\subseteq [b]^k$ satisfies
\[
\abs{C}\le \sum_{t=2}^b B_t(k).
\]
Indeed, a nontrivial clique contains no word of support size one, and
\Cref{thm:circular-support} bounds the contribution of each support size
$t\ge2$; the singleton case is trivial.

This summed circular bound is eventually stronger than the path bound in every
nondivisible residue class: the full-support term is smaller by the factor
$bq/k$, while all smaller-support terms are exponentially smaller than the
full-support term for fixed $b$.  Still, we expect the lower-bound construction
to be the exact answer for all congruence classes once $k$ is large enough.  The
obstacle in this argument is that it treats the support sizes independently and
therefore carries additive lower-support terms.

\subsection{The top residue for \texorpdfstring{$b=3$}{b=3}}

Put $k\equiv 2\pmod 3$ and $q=(k-2)/3$.  Then
\Cref{thm:circular-support} with $t=3$ gives
\[
\abs{C_3}\le
\frac{q(k-1)!}{2q!(q+1)!^2}.
\]
This equals
\[
N_3(k)=\binom{k-1}{q-1}\binom{2q+1}{q},
\]
so (as we have already seen for $k \equiv -1 \pmod b$) 
the canonical construction is sharp on the full-support
 part in this residue
class.

An arbitrary clique $C\subseteq [3]^k$ decomposes as $C=C_2\cup C_3$, since a
support-one word cannot belong to a nontrivial clique.  The summed support bound
would therefore add a separate quantity for $\abs{C_2}$ to the full-support
estimate above.  Since the construction 
itself has full support, one naturally
expects this extra term to be unnecessary.  The following theorem proves
this in the present case: 
for sufficiently large $q$, an arbitrary clique is
already bounded by the full-support quantity above.

The proof is a stability refinement of the same cyclic-order count: a
support-two word is charged against the deficit it forces in the full-support
circular packing.  This is analogous in spirit to Hilton--Milner type
refinements of Erd\H{o}s--Ko--Rado and to later applications of Katona's circle
method \cite{HiltonMilner,FranklKatonaCircle}.

\begin{theorem}\label{thm:b3-top-residue-exact}
Let $k\equiv 2\pmod 3$, and put $q=(k-2)/3$.  If $q\ge 5$, then
\[
m_3(k)=N_3(k).
\]
More precisely, every clique $C\subseteq [3]^k$ satisfies
\[
\abs{C}\le N_3(k),
\]
and if $C_2\ne\varnothing$, then the inequality is strict.
\end{theorem}

The proof is given in \Cref{app:b3-support-two}.

\section{Concluding Remarks and Open Problems}\label{sec:discussion}

The main open question is whether the canonical construction is eventually
optimal for every fixed alphabet size.  In other words, we expect that for each
fixed $b$,
\[
m_b(k)=N_b(k)
\]
for all sufficiently large $k$.  The results above prove this asymptotically,
prove it exactly when $k$ is large and divisible by $b$, and prove it exactly
for $b=3$ when $k\equiv -1\pmod 3$ and $k$ is sufficiently large.

The circular argument suggests a route to the general conjecture.  For
full-support cliques it gives the sharp bound when $k\equiv0,-1\pmod b$, but
the unrestricted bound obtained by summing over support sizes still includes
additive lower-support terms.  We believe these terms are artifacts of treating
the support sizes separately.  The proof for $b=3$ and $k\equiv-1\pmod3$ shows
one possible mechanism: a lower-support word forces enough deficit in the
full-support circular packing to pay for itself.  A natural next step is to
develop such a stability argument for general $b$.

The other residue classes may require a sharper fixed-support argument.  When
$1\le r\le b-2$, the circular full-support bound remains larger than the
construction by an explicit factor, so removing the additive lower-support
terms alone would not prove the conjectured exact value.  It would be useful to
find a refinement of the circular packing, perhaps using additional boundary
information, that closes this remaining gap.

One can try to extend the upper bounds proved here for $m_b(k)$ for
pairs that are $\ell$-pseudo-shattered (for $\ell  \geq 1$) 
in the sense of \cite{HannekeMengMoranShaeiri}.  
In dimension two, this means
that the support graph contains a nonempty subgraph of minimum 
degree at least
$\ell+1$.  Let $m_{b,\ell}(k)$ be the maximum size of a family
$C\subseteq [b]^k$ with this property for every pair of distinct 
words; thus
$m_{b,1}(k)=m_b(k)$.

It is possible to replace the path and cycle permutation packing
arguments used for $\ell=1$ by embedding in other appropriate graphs. 
While this gives some nontrivial bounds they appear to be far from
optimal. It is worth noting that for
$\ell=b-1$, the condition is
qualitative independence, whose sharp asymptotic exponent 
is known to be $2/b$, as proved by Gargano,
K\"orner, and Vaccaro \cite{GarganoKornerVaccaro}.

\appendix
\section{The Support-Two Replacement Argument for \texorpdfstring{$b=3$}{b=3}}\label{app:b3-support-two}

\begin{proof}[Proof of \Cref{thm:b3-top-residue-exact}]
The canonical construction gives the lower bound, so it remains to prove the
upper bound.  Fix a clique $C\subseteq [3]^k$.  If $\abs{C}=1$, the result is
trivial, so assume $\abs{C}>1$.  Then $C$ contains no constant word, because a
constant word cannot form a cycle with any other word.  Thus
\[
C=C_2\cup C_3.
\]

\medskip
\noindent
\emph{The full-support defect.}
Let
\[
w=2q!(q+1)!^2,\qquad N=N_3(k)=\frac{q(k-1)!}{2q!(q+1)!^2}.
\]
For a full-support word $x$, let $p(x)$ be the number of oriented cyclic orders
$\Omega$ for which $x$ is an $\Omega$-interval word.  If the three multiplicities
of $x$ are $a,b,c$, then
\[
p(x)=2a!b!c!.
\]
Indeed, there are $2=(3-1)!$ cyclic orders of the three labelled level sets, and
then the elements inside each level set may be ordered arbitrarily.

The product $a!b!c!$, subject to $a+b+c=k$ and $a,b,c>0$, is minimized at the
balanced multiplicities
\[
(q,q+1,q+1).
\]
Thus every $x\in C_3$ satisfies
\[
p(x)\ge w,
\]
with equality exactly for balanced full-support words.

For a cyclic order $\Omega$, put
\[
a_3(\Omega)=\abs{\{x\in C_3:x\text{ is an }\Omega\text{-interval word}\}}.
\]
By \Cref{lem:circular-boundaries-disjoint}, the boundary triples
$\partial_\Omega x$, for $x\in C_3$ an $\Omega$-interval word, are pairwise
disjoint $3$-subsets of the $k$ cycle-edges.  Thus
\[
a_3(\Omega)\le q.
\]
Summing over all $(k-1)!$ oriented cyclic orders gives
\[
\sum_{x\in C_3}p(x)\le q(k-1)!.
\]
Since $wN=q(k-1)!$, this already gives
\[
\abs{C_3}\le N.
\]

Define
\[
\Delta=N-\abs{C_3}\ge 0,
\qquad
E=\sum_{x\in C_3}(p(x)-w),
\qquad
D(\Omega)=q-a_3(\Omega).
\]
Since $wN=q(k-1)!$, we have
\[
\sum_\Omega D(\Omega)
=q(k-1)!-\sum_{x\in C_3}p(x)
=w\Delta-E.
\]
In particular, $E\le w\Delta$.

Call a full-support word balanced if its three multiplicities are
$q,q+1,q+1$ in some order, and call it unbalanced otherwise.  Set
\[
U(\Omega)=
\abs{\{x\in C_3:x\text{ is unbalanced and }\Omega\text{-interval}\}}.
\]

\begin{lemma}\label{lem:b3-unbalanced-cost}
We have
\[
\sum_\Omega U(\Omega)\le (q+2)E.
\]
Consequently,
\[
\sum_\Omega\bigl(D(\Omega)+U(\Omega)\bigr)\le (q+2)w\Delta.
\]
\end{lemma}

\begin{proof}
Among all unbalanced triples of positive integers with sum $k$, the product
of factorials is minimized at $(q,q,q+2)$: if two parts differ by at least $3$,
moving one unit from the larger part to the smaller part decreases the product.
Therefore every unbalanced $x\in C_3$ satisfies
\[
p(x)\ge 2q!q!(q+2)!=w\frac{q+2}{q+1}.
\]
Hence
\[
p(x)\le (q+2)(p(x)-w).
\]
Summing over the unbalanced full-support words gives
\[
\sum_\Omega U(\Omega)\le (q+2)E.
\]
Combining this with $\sum_\Omega D(\Omega)=w\Delta-E$ and $E\le w\Delta$ gives
\[
\sum_\Omega\bigl(D(\Omega)+U(\Omega)\bigr)
\le (w\Delta-E)+(q+2)E
=w\Delta+(q+1)E
\le (q+2)w\Delta.
\]
\end{proof}

\medskip
\noindent
\emph{A support-two word forces many bad cyclic orders.}
Now fix $z\in C_2$.  Let its two nonempty level sets be $A$ and $B$, with
\[
\abs{A}=a\le \abs{B}.
\]
Let $p(z)$ be the number of oriented cyclic orders for which $z$ is an interval
word.  Then
\[
p(z)=a!(k-a)!.
\]
In particular,
\[
p(z)\ge \mu,
\qquad
\mu=\floor{k/2}!\ceil{k/2}!.
\]
Call a cyclic order $\Omega$ \emph{bad for $z$} if $z$ is an $\Omega$-interval
word and
\[
D(\Omega)\ge 1\qquad\text{or}\qquad U(\Omega)\ge 1.
\]

\begin{proposition}\label{prop:b3-half-bad}
For every $z\in C_2$,
\[
\abs{\{\Omega:\Omega\text{ is bad for }z\}}\ge \frac{p(z)}2.
\]
\end{proposition}

\begin{proof}
We split into two cases.

\medskip
\noindent
\textbf{Case 1: $a\le q$.}
Let $\Omega$ be any cyclic order in which $z$ is an interval word.  The two
boundary edges of $z$ split the circle into two arcs, one corresponding to $A$
and the other to $B$.  The $A$-arc has only $a-1\le q-1$ internal cycle-edges.

Let $x\in C_3$ be an $\Omega$-interval word.  If $x$ is cyclic with $z$, then at
least two level sets of $x$ must meet both $A$ and $B$; otherwise the support
graph $G(z,x)$, which has only two vertices on the $z$-side, cannot contain a
cycle.  Therefore $\partial_\Omega x$ must contain an internal edge of the
$A$-arc.

By \Cref{lem:circular-boundaries-disjoint}, these boundary edges are distinct
for distinct $x$.  Hence at most $a-1\le q-1$ full-support words of $C$ can be
$\Omega$-interval words.  Thus $D(\Omega)\ge1$, so every interval order of $z$
is bad.

\medskip
\noindent
\textbf{Case 2: $a\ge q+1$.}
Write
\[
a=q+s,\qquad 1\le s\le q.
\]
The upper bound follows from $a\le \abs{B}=3q+2-a$.

We first record an exact-cover calculation.  Index the cycle-edges by
\[
0,1,\dots,3q+1\pmod{3q+2},
\]
and suppose the two boundary edges of $z$ are
\[
q+s-1\qquad\text{and}\qquad 3q+1.
\]
A balanced full-support interval word has block sizes $q,q+1,q+1$.  Therefore
its boundary triple has the form
\[
T_i=\{i,\ i+q,\ i+2q+1\}\pmod{3q+2}
\]
for a unique $i$.

Suppose $q$ such triples are pairwise disjoint and avoid the two boundary edges
of $z$.  Since they use $3q$ edges, they cover every edge except
\[
q+s-1,\qquad 3q+1.
\]
Let $f_i\in\{0,1\}$ indicate whether $T_i$ is selected.  Each edge $j$ belongs
exactly to the three triples $T_j,T_{j-q},T_{j-2q-1}$, so
\[
f_j+f_{j-q}+f_{j-2q-1}
=
\begin{cases}
0,& j=q+s-1\text{ or }j=3q+1,\\
1,& \text{otherwise},
\end{cases}
\tag{*}
\]
with all indices taken modulo $3q+2$.  The equation at $3q+1$ gives
\[
f_{3q+1}=f_{2q+1}=f_q=0
\]
and the equation at $q+s-1$ gives
\[
f_{q+s-1}=f_{s-1}=f_{2q+s}=0.
\]
Starting from these zeros and applying $(*)$ successively along the two chains
gives
\[
f_0=f_1=\cdots=f_{s-2}=1
\]
and
\[
f_{q+s}=f_{q+s+1}=\cdots=f_{2q}=1,
\]
while all other $f_i$'s are forced to be $0$.  Thus the selected indices are
exactly
\[
\{0,1,\dots,s-2\}\cup \{q+s,q+s+1,\dots,2q\}.
\tag{1}
\]

We now pair the interval orders of $z$.  A display
\[
[X_1,X_2,\dots,X_m]
\]
means the cyclic order whose consecutive blocks are $X_1,X_2,\dots,X_m$, with
the inherited internal orders.

\medskip
\noindent
\emph{Subcase $s=1$.}
Here $\abs{A}=q+1$ and $\abs{B}=2q+1$.  Every interval cyclic order of $z$ can
be uniquely written as
\[
\Omega=[A_0,a_\ast,V,U,b_\ast],
\]
where
\[
\abs{A_0}=q,\qquad \abs{U}=\abs{V}=q,
\]
$a_\ast\in A$, and $b_\ast\in B$.  The $A$-block is $A_0,a_\ast$, and the
$B$-block is $V,U,b_\ast$.

Define a bijection on interval orders of $z$ by
\[
\phi(\Omega)=[A_0,b_\ast,V,U,a_\ast].
\]
The image still has $A$ and $B$ as cyclic intervals: the $A$-block wraps around
the end of the displayed cyclic order.

We claim that $\Omega$ and $\phi(\Omega)$ cannot both be not bad.  Suppose they
were both not bad.  Then in each order there are exactly $q$ full-support
interval words of $C$, all balanced.  By the exact-cover calculation above, for
\[
\Omega=[A_0,a_\ast,V,U,b_\ast],
\]
the forced family of boundary triples includes the triple whose blocks are
\[
P=\{\,U,\ A_0\cup\{b_\ast\},\ \{a_\ast\}\cup V\,\}.
\]
Thus $P$ is the level-set partition of some word of $C_3$.

Similarly, after
rotating
\[
\phi(\Omega)=[A_0,b_\ast,V,U,a_\ast]
\]
to the display
\[
[a_\ast,A_0,b_\ast,V,U],
\]
the exact-cover calculation forces the partition
\[
Q=\{\,V,\ \{a_\ast\}\cup U,\ A_0\cup\{b_\ast\}\,\}
\]
to occur in $C_3$.

But the overlap graph of $P$ and $Q$ is a forest.  Indeed, the common block
$A_0\cup\{b_\ast\}$ gives an isolated edge component, and the remaining nonzero
intersections form the path
\[
U
-
(\{a_\ast\}\cup U)
-
(\{a_\ast\}\cup V)
-
V.
\]
Therefore the corresponding two words are not cyclic, contradicting that $C$ is
a clique.  Hence at least one of $\Omega,\phi(\Omega)$ is bad.  Since $\phi$ is
a bijection, at least half of the interval orders of $z$ are bad.

\medskip
\noindent
\emph{Subcase $s\ge 2$.}
Every interval cyclic order of $z$ can be uniquely written as
\[
\Omega=[U_A,a_0,A',B',U_B,b_0],
\]
where
\[
\abs{U_A}=q,\qquad \abs{A'}=s-1,\qquad
\abs{U_B}=q,\qquad \abs{B'}=q+1-s.
\]
The $A$-block is $U_A,a_0,A'$, and the $B$-block is $B',U_B,b_0$.

Define a bijection by
\[
\phi(\Omega)=[U_A,A',B',b_0,U_B,a_0].
\]
Again, the image has $A$ and $B$ as cyclic intervals, with the $A$-block
wrapping around the displayed end.

Suppose $\Omega$ and $\phi(\Omega)$ are both not bad.  For
\[
\Omega=[U_A,a_0,A',B',U_B,b_0],
\]
the exact-cover calculation forces the partition
\[
P=\{\,U_B,\ U_A\cup\{b_0\},\ \{a_0\}\cup A'\cup B'\,\}
\]
to occur in $C_3$.

For $\phi(\Omega)$, rotate the display to
\[
[a_0,U_A,A',B',b_0,U_B].
\]
The exact-cover calculation then forces
\[
Q=\{\,U_A,\ \{a_0\}\cup U_B,\ A'\cup B'\cup\{b_0\}\,\}
\]
to occur in $C_3$.

The overlap graph of $P$ and $Q$ is the path
\[
U_B
-
(\{a_0\}\cup U_B)
-
(\{a_0\}\cup A'\cup B')
-
(A'\cup B'\cup\{b_0\})
-
(U_A\cup\{b_0\})
-
U_A.
\]
Hence the two corresponding words are not cyclic, contradiction.  Therefore at
least one of $\Omega,\phi(\Omega)$ is bad.  Since $\phi$ is a bijection, at
least half of the interval orders of $z$ are bad.

This completes the proof of the proposition.
\end{proof}

\medskip
\noindent
\emph{Counting bad incidences.}
Let
\[
\mathcal I=\{(z,\Omega):z\in C_2,\ \Omega\text{ is bad for }z\}.
\]
By \Cref{prop:b3-half-bad},
\[
\abs{\mathcal I}\ge \abs{C_2}\frac{\mu}{2}.
\]

We now upper-bound $\abs{\mathcal I}$.  Fix a cyclic order $\Omega$.  The
support-two words in $C_2$ that are $\Omega$-interval words have pairwise
disjoint boundary pairs, by \Cref{lem:circular-boundaries-disjoint}.  Therefore
there are at most
\[
M=\floor{k/2}
\]
of them.  Moreover, if $\Omega$ is bad for some $z$, then
\[
D(\Omega)+U(\Omega)\ge1.
\]
Thus, using \Cref{lem:b3-unbalanced-cost},
\[
\abs{\mathcal I}
\le M\sum_\Omega\bigl(D(\Omega)+U(\Omega)\bigr)
\le M(q+2)w\Delta.
\]
Consequently
\[
\abs{C_2}\le \frac{2M(q+2)w}{\mu}\,\Delta.
\]

\medskip
\noindent
\emph{The final numerical estimate.}
For $q\ge 5$ we have
\[
\mu>2M(q+2)w.
\]
Indeed, this is equivalent to
\[
R_q:=
\frac{\floor{(3q+2)/2}!\ceil{(3q+2)/2}!}
{4\floor{(3q+2)/2}(q+2)q!(q+1)!^2}
>1.
\]
Direct computation gives $R_5=21/20$ and $R_6=9/4$.  If
\[
n=\floor{(3q+2)/2},\qquad m=\ceil{(3q+2)/2},
\]
then replacing $q$ by $q+2$ increases both $n$ and $m$ by $3$, and cancellation
gives
\[
\frac{R_{q+2}}{R_q}
=
\frac{n(n+1)(n+2)(m+1)(m+2)(m+3)}
{(q+1)(q+2)^2(q+3)^2(q+4)}.
\]
For $q\ge 5$ each numerator factor is larger than the corresponding denominator
factor, so $R_{q+2}>R_q$.  Thus $R_q>1$ for all $q\ge 5$.

It follows that $\abs{C_2}<\Delta$ whenever $\Delta>0$, while if $\Delta=0$ then
$\abs{C_2}=0$.  Hence
\[
\abs{C}=\abs{C_3}+\abs{C_2}=N-\Delta+\abs{C_2}\le N.
\]
If $C_2\ne\varnothing$, then $\Delta>0$ and $\abs{C_2}<\Delta$, so the inequality is
strict.  The canonical construction gives the matching lower bound.
\end{proof}
\vspace{0.2cm}

\noindent
{\bf Acknowledgment:}\, Some of the proofs in this paper were found with the
help of ChatGPT 5.5 Pro and were subsequently verified by the authors.

\end{document}